\documentclass[11pt]{amsart}
\input epsf
\usepackage{latexsym}
\usepackage{amssymb,amsmath,amsthm,amscd, amssymb,
graphicx, enumerate, verbatim}
\usepackage[all]{xy}

\numberwithin{equation}{section}

\newtheorem{lem}[equation]{Lemma}
\newtheorem{prop}[equation]{Proposition}
\newtheorem{thm}[equation]{Theorem}

\newtheorem{Example}[equation]{Example}

\newtheorem{remark}[equation]{Remark}
\newenvironment{rmk}{\begin{remark}\rm}{\end{remark}}
\def\co{\colon\thinspace}

\def\a{\alpha}

\def\g{\gamma}

\def\d{\partial}

\def\s{\sigma}
\def\l{\lambda}

\def\S1{\bf S^1}

\makeatletter
\def\equalsfill{$\m@th\mathord=\mkern-7mu
\cleaders\hbox{$\!\mathord=\!$}\hfill
\mkern-7mu\mathord=$}
\makeatother

\begin{document}

\abovedisplayskip=6pt plus3pt minus3pt
\belowdisplayskip=6pt plus3pt minus3pt

\title[Complete nonnegatively curved planes]
{\bf Connectedness properties 
of the space of complete nonnegatively curved planes}

\rm

\author{Igor Belegradek}

\address{Igor Belegradek\\School of Mathematics\\ Georgia Institute of
Technology\\ Atlanta, GA 30332-0160}\email{ib@math.gatech.edu}

\author{Jing Hu}

\address{Jing Hu\\School of Mathematics\\ Georgia Institute of
Technology\\ Atlanta, GA 30332-0160}\email{jhu61@math.gatech.edu}


\date{}
\begin{abstract}
We study the space of complete Riemannian metrics of nonnegative curvature
on the plane equipped with the $C^k$ topology.
If $k$ is infinite, we show that the space is homeomorphic to the 
separable Hilbert space. For any $k$ we prove that the space
cannot be made disconnected by removing a finite dimensional subset. 
A similar result holds for the associated moduli space. The proof combines 
properties of subharmonic functions with results
of infinite dimensional topology and dimension theory.
A key step is a characterization of the conformal factors that make
the standard Euclidean metric on the plane
into a complete metric of nonnegative sectional curvature.  
\end{abstract}
\maketitle

\section{Introduction}

Spaces of constant curvature metrics on surfaces 
is the subject of Teichm\"uller theory. 
The spaces of Riemannian metrics (and the associated moduli spaces) 
have been studied under various geometric assumptions such as
positive scalar \cite{KS-mod, Ros-surv, Mar-3sph-ric-fl}, 
negative sectional \cite{FarOnt-ann09, FarOnt-jtop10, FarOnt-jdg10, FarOnt-gafa10}, 
positive Ricci \cite{Wra-ric-mod}, nonnegative 
sectional \cite{KPT-mod, BKS-mod1, BKS-mod2, Ott-non-homeo},
while curvature-free results about spaces of metrics can be found 
in \cite{Ebi, NabWei-ihes}.

Throughout the paper the topology of 
$C^k$ uniform convergence on compact sets, where $0\le k\le \infty$,
will be called {\it the $C^k$ topology}.

This paper studies connectedness properties of the 
set of complete nonnegatively curved metrics
on $\mathbb R^2$ equipped with the $C^k$
topology. The main theme is deciding 
when two metrics can be deformed to each other 
through complete nonnegatively curved metrics 
outside a given subset, and how large 
the space of such deformations is.

The starting point is a result of Blanc-Fiala~\cite{BlaFia} that
any complete nonnegatively curved metric on $\mathbb R^2$ is 
conformally equivalent to the standard Euclidean metric $g_0$, i.e.
isometric to $e^{-2u} g_0$ for some smooth function $u$,
see~\cite[Theorem 15]{Hub-1957} 
and~\cite[Corollary 7.4]{Gri-surv} for generalizations.

The sectional curvature of $e^{-2u} g_0$ equals $e^{2u}\Delta u$, 
where $\triangle$ is the Euclidean Laplacian.
Thus $e^{-2u} g_0$ has nonnegative curvature 
if and only if $u$ is subharmonic.
Characterizing subharmonic functions 
that correspond to complete metrics is not straightforward, and 
doing so is the main objective of this paper. 
A basic property~\cite[Theorem 2.14]{HayKen-vol1} 
of a subharmonic function $u$ on $\mathbb R^2$ is that the limit
\[
\a(u):=\displaystyle{\lim_{r\to \infty}\frac{M(r,u)}{\log r}}
\]
exists in $[0,\infty]$, where $M(r,u):=\sup\{u(z):|z|=r\}$. 
For example, 
by Liouville's theorem $\a(u)=0$ if and only if $u$ is constant, while
any nonconstant harmonic function $u$ satisfies
$\a(u)=\infty$ (see Proposition~\ref{prop: harmonic}). 

Completeness of $e^{-2u}g_0$ follows easily when $\a(u)<1$.
Appealing to more delicate properties of subharmonic functions due to
Huber~\cite{Hub-1957} and Hayman~\cite{Hay-slow} we prove:

\begin{thm}
\label{thm: char of completeness}
The metric $e^{-2u}g_0$ is complete 
if and only if $\a(u)\le 1$. 
\end{thm}

Note that complete Riemannian
metrics on any manifold form a dense subset in the space of all Riemannian metrics, 
e.g. $e^{-2u}g_0$ is the endpoint of the
curve $g_s:=(s+e^{-2u})g_0$ where the metric $g_s$ is complete for $s>0$~\cite{FegMil}.

Denote the set of $C^\infty$ complete Riemannian metrics on $\mathbb R^2$
of nonnegative sectional curvature equipped with
the $C^k$ topology by $\mathcal R^{k}_{\ge 0}(\mathbb R^2)$.
The space $\mathcal R^{k}_{\ge 0}(\mathbb R^2)$ 
is metrizable and separable,
see Lemma~\ref{lem: separable}. We prove

\begin{thm}
\label{intro-thm: space of metrics homeo l2}
$\mathcal R^{\infty}_{\ge 0}(\mathbb R^2)$ is homeomorphic to $\ell^2$, the separable Hilbert space. 
\end{thm}

Any metric conformally equivalent to the standard metric on
$\mathbb R^2$ can be written {\it uniquely\,} as $\varphi^*e^{-2u} g_0$
where $g_0$ is the standard Euclidean metric, $u$ is a smooth function, 
and $\varphi\in\mathrm{Diff}^+_{0,1}(\mathbb R^2)$,
the group of self-diffeomorphisms of the plane fixing the complex numbers 
$0$, $1$ and isotopic to the identity, see Lemma~\ref{lem: conformal unique}. 

Let $S_\a$ be the subset of $C^\infty(\mathbb R^2)$ consisting of
subharmonic functions with $\a(u)\le \a$. 
By Theorem~\ref{thm: char of completeness}
the map $(u, \varphi)\to \varphi^*e^{-2u}g_0$ defines
a bijection
\begin{equation}
\label{form: bijection}
\Pi\co S_1\times\mathrm{Diff}^+_{0,1}(\mathbb R^2)\to 
\mathcal R^k_{\ge 0}(\mathbb R^2).
\end{equation} 

We will show in Theorem~\ref{thm: pi k homeo} that 
$\Pi$ becomes a homeomorphism if $S_1$ and $\mathcal R^k_{\ge 0}(\mathbb R^2)$ are 
given the $C^{k+\g}$ topology and $\mathrm{Diff}^+_{0,1}(\mathbb R^2)$ is given the $C^{k+1+\g}$ topology,
where $\g\in (0,1)$.
The idea of the proof is to write $\varphi^*e^{-2u}g_0$ as $\l|dx+\mu d\bar z|$,
note that $\phi$ solves the Beltrami equations $\varphi_{\bar z}=\mu\varphi_{z}$,
and use the fact that solutions of the Beltrami equation depend smoothly on $\mu$.
For finite $k$ we do not yet know the homeomorphism type of 
$S_1$, $\mathrm{Diff}^+_{0,1}(\mathbb R^2)$ in the $C^{k+\g}$, $C^{k+1+\g}$ topologies,
respectively; we hope to address this in future work.

Unless stated otherwise we equip 
$S_\a$, $C^\infty(\mathbb R^2)$, and
$\mathrm{Diff}^+_{0,1}(\mathbb R^2)$ with the $C^\infty$ topology.
In this topology the bijection $\Pi$ is clearly continuous.
Theorem~\ref{thm: pi k homeo} stated above implies that $\Pi$ is a homeomorphism for $k=\infty$.
By contrast, if $k$ is finite, then $\Pi$ is not a homeomorphism,
because it factors as the composite of
$\Pi\co S_1\times\mathrm{Diff}^+_{0,1}(\mathbb R^2)\to \mathcal R^{\infty}_{\ge 0}(\mathbb R^2)$ and
$\mathbf{id}\co\mathcal R^{\infty}_{\ge 0}(\mathbb R^2)\to\mathcal R^{k}_{\ge 0}(\mathbb R^2)$ 
and the latter map is a continuous bijection that is clearly not a homeomorphism.

The $C^\infty$ topology makes $C^\infty(\mathbb R^2)$ into a separable Fr\'echet space, 
see Lemma~\ref{lem: separable}.   
Moreover, we show in Lemmas~\ref{lem: S_alpha closed}--\ref{lem: S_alpha} that the subset $S_\a$
of $C^\infty(\mathbb R^2)$ is closed, convex, and not locally compact when $\a\neq 0$.
Since $S_0$ consists of constants, it is homeomorphic to $\mathbb R$.

What makes the continuous bijection $\Pi$ useful is the fact that 
the parameter space $S_1\times\mathrm{Diff}^+_{0,1}(\mathbb R^2)$ is homeomorphic 
to $\ell^2$, which is implied by the following
results of infinite dimensional topology:
\begin{itemize}
\item
any closed convex non-locally-compact subset of a separable Fr\'echet space is 
homeomorphic to $\ell^2$~\cite[Theorem 2]{DobTor}; e.g.
this applies to $S_\a\subset C^\infty(\mathbb R^2)$ with $\a\neq 0$.
\item
$\mathrm{Diff}^+_{0,1}(\mathbb R^2)$ is homeomorphic to $\ell^2$~\cite[Theorem 1.1]{Yag}.
\item
$\ell^2$ is homeomorphic to $(-1,1)^{\mathbb N}$, the countably infinite product of 
open intervals~\cite{And-l2-lines}, so $\ell^2\times \ell^2$ is homeomorphic to $\ell^2$.
\end{itemize}
In particular, the above discussion yields Theorem~\ref{intro-thm: space of metrics homeo l2}.

Our first application demonstrates that any two metrics can be 
deformed to each other in a variety of ways, while bypassing a given countable set:

\begin{thm} 
\label{thm: space of metrics countable}
If $K$ is a countable subset of $\mathcal R^{k}_{\ge 0}(\mathbb R^2)$
and $X$ is a separable metrizable space, then for 
any distinct points $x_1, x_2\in X$ and any distinct metrics $g_1, g_2$ 
in $\mathcal R^{k}_{\ge 0}(\mathbb R^2)\smallsetminus K$
there is an embedding of $X$ into $\mathcal R^{k}_{\ge 0}(\mathbb R^2)\smallsetminus K$
that takes $x_1$, $x_2$ to $g_1$, $g_2$, respectively.
\end{thm}

Some deformations in $\mathcal R^{k}_{\ge 0}(\mathbb R^2)$ can be constructed explicitly
(e.g. one could slightly change the metric near a point where $K>0$, or one could
join two embedded convex surfaces in $\mathbb R^3$ by the path of their convex combination).
Yet it is unclear how such methods could yield Theorem~\ref{thm: space of metrics countable}.
Instead we use infinite dimensional topology. The above theorem 
is an easy consequence of the following facts:
\begin{itemize}
\item
Like any continuous one-to-one map to a Hausdorff space, the map $\Pi$ 
restricts to a homeomorphism on every compact subset, e.g.
the Hilbert cube.
\item
Every separable metrizable space embeds into the Hilbert 
cube~\cite[Theorem IX.9.2]{Dug}. 
\item
The complement in $\ell^2$ of the countable union of compact
sets is homeomorphic to $\ell^2$~\cite{And-defic}, 
cf.~\cite[Theorem V.6.4]{BP-book}),
and hence contains an embedded Hilbert cube.
\end{itemize} 

A topological space is {\it continuum-connected\,} if every two points
lie in a {\it continuum\,} (a compact connected space); thus 
a continuum-connected space is connected but not necessarily path-connected.

By {\it dimension\,} we mean the covering dimension, 
see~\cite{Eng-book} for background. Note that for separable metrizable
spaces the covering dimension equals the small and the large inductive dimensions, 
see~\cite[Theorem 1.1.7]{Eng-book}.
A separable metric space is finite dimensional 
if and only if it embeds into a Euclidean space~\cite[Theorems 1.1.2 and 1.11.4]{Eng-book}. 

By Theorem~\ref{thm: space of metrics countable} any two metrics in 
$\mathcal R^{k}_{\ge 0}(\mathbb R^2)$
lie in an embedded copy of $\mathbb R^n$ for any $n$. 
The fact that $\mathbb R^n$ cannot be separated by a subset of codimension $\ge 2$
easily implies the following.

\begin{thm} \label{thm: space of metrics finite dim}
The complement of every finite dimensional subset of $\mathcal R^{k}_{\ge 0}(\mathbb R^2)$
is continuum-connected. 
The complement of every closed finite dimensional subset of $\mathcal R^{k}_{\ge 0}(\mathbb R^2)$
is path-connected.
\end{thm}

Let  $\mathcal M^{k}_{\ge 0}(\mathbb R^2)$ denote
{\it the moduli space of complete nonnegatively curved metrics},
i.e. the quotient space of $\mathcal R^{k}_{\ge 0}(\mathbb R^2)$
by the $\mathrm{Diff}(\mathbb R^2)$-action via pullback.
The moduli space $\mathcal M^{k}_{\ge 0}(\mathbb R^2)$ is rather pathological, e.g. 
it is not a Hausdorff space (in Proposition~\ref{prop: cone, not homeo}
we exhibit a non-flat metric $g\in \mathcal R^{k}_{\ge 0}(\mathbb R^2)$ 
whose isometry class lies in every neighborhood of
the isometry class of $g_0$). Consider
the map $S_1\to \mathcal M^{k}_{\ge 0}(\mathbb R^2)$
sending $u$ to the isometry class of $e^{-2u}g_0$. Its 
fibers lie in the orbits of a $\mathrm{Conf}(g_0)$-action of $C^\infty(\mathbb R^2)$,
so each fiber is the union of countably many finite dimensional
compact sets, which by dimension theory arguments easily implies:

\begin{thm}\label{thm: mod disconnected}
The complement of a subset $S$ of $\mathcal M^{k}_{\ge 0}(\mathbb R^2)$
is path-connected if $S$ is countable, or if $S$
is closed, metrizable and finite dimensional.
\end{thm}

The above applications generalize in two different directions:

\begin{enumerate}
\item
The deformations in 
Theorems~\ref{thm: space of metrics countable}--\ref{thm: mod disconnected} 
can be arranged to bypass any given set of complete
flat metrics (see Remark~\ref{rmk: flat metrics}). 
\item
Replacing $S_1$ with $S_\infty$ in the
proofs of Theorems~\ref{thm: space of metrics countable}--\ref{thm: mod disconnected} 
one gets the same conclusions with
$\mathcal R^{k}_{\ge 0}(\mathbb R^2)$ 
substituted by the subspace of $\mathcal R^{k}(\mathbb R^2)$
of nonnegatively curved metrics that are conformally equivalent to $g_0$. 
\end{enumerate}

That $\mathcal R^{k}_{\ge 0}(\mathbb R^2)$ and $\mathcal M^{k}_{\ge 0}(\mathbb R^2)$
cannot be separated by a countable set was announced by the first author
in~\cite{Bel-ober} who at that time 
only knew that
$e^{-2u}g_0$ is complete when $\a(u)<1$ and incomplete for $\a(u)>1$.
Completeness of $e^{-2u}g_0$  when $\a(u)=1$ was established
in collaboration with the second author, which led to stronger
applications.

Theorem~\ref{thm: char of completeness} is proved in Section~\ref{sec: subharm}.
Some auxiliary results are collected in Section~\ref{sec: misc}, while 
proofs of Theorems~\ref{thm: space of metrics countable}--\ref{thm: mod disconnected}
are given in Section~\ref{sec: applications}.

{\bf Acknowledgments:} 
Belegradek is grateful for NSF support (DMS-1105045).
We are thankful to the referee
for expository suggestions, and for spotting a mistake in the previous 
version which lead us to 
discover Theorem~\ref{thm: pi k homeo}.

\section{Subharmonic functions and complete metrics}
\label{sec: subharm}

We adopt notation of the introduction:
$g_0$ is the standard Euclidean metric, $u$
is a smooth subharmonic function on $\mathbb R^2$,
and $M(r, u):=\max\{u(z) : |z|=r\}$. Set 
$\mu(t, u):=M(e^t, u)$; 
when $u$ is understood we simply write $M$, $\mu$.
Subharmonicity of $u$ implies that $\mu$ is 
a convex function~\cite[Theorem 2.13]{HayKen-vol1}.
Hence $\mu$ has left and right derivatives everywhere,
and they are equal outside a countable subset,
and so the same holds for $M$.
By the maximum principle~\cite[Theorem 2.3]{HayKen-vol1} 
$M$ is strictly increasing (except when $u$ is constant), and
hence the same is true for $\mu$.
As mentioned in the introduction, the limit \[
\a(u):=\lim_{r\to \infty} \frac{M(r, u)}{\log r}\] exists 
in $[0,\infty]$, 
and $\a(u)=0$ if and only if $u$ is constant~\cite[Theorem 2.14]{HayKen-vol1}. 

It follows from the Hopf-Rinow theorem
that a Riemannian manifold is incomplete if and only if
it contains a locally rectifiable (or equivalently, a smooth) path that eventually leaves every compact 
set and has finite length. (Indeed, an incomplete manifold is geodesically incomplete,
and hence has a finite length geodesic that leaves compact subsets. Conversely, given a rectifiable 
path as above, choose a sequence of points on the path that escapes compact subsets,
and note that the sequence is bounded because the path has finite length.)

In the manifold $(\mathbb R^2, e^{-2u}g_0)$ the length of a path $\g$ 
equals $\int_\g e^{-u} ds$.

\begin{lem} 
\label{lem: complete > and <}
The metric $g=e^{-2u}g_0$ is
complete if  $a(u)<1$ and incomplete if $\a(u)>1$ or $\a(u)=\infty$.
\end{lem}
\begin{proof}
If $\a(u)=\infty$, then incompleteness of $g$ can be extracted 
from~\cite{Hub-1957}, see~\cite{LRW} for a stronger result, 
who constructed a finite length locally rectifiable path that
goes to infinity in $\mathbb R^2$.  
If $\a(u)=0$, then $u$ is constant, so $g$ is complete. 

So we can assume that $\a(u)$ is positive and finite in which case 
Hayman~\cite[Theorem 2 and Remark (i) on page 75]{Hay-slow} proved
that there is a constant $c$ and a measure zero subset $Z$
of the unit circle such that $0\le M(r)-u(re^{i\theta})\le c$ 
for every $\theta\notin Z$ and all $r>r(\theta)$.

Suppose $\a(u)>1$, and fix $\theta\notin Z$,
and the corresponding ray $\g(r)=re^{i\theta}$, $r>r(\theta)$ 
on which $0\le M(r)-u(re^{i\theta})\le c$. 
Then $\int_\g e^{-u}$ is bounded above and below by positive multiples of
$\int_\g e^{-M}$. 
As $\frac{M(r)}{\log r}\to \a$, for any $\a_0\in (1,\a(u))$ there is $r_0$
with $M(r)>\a_0\log r$ for all $r>r_0$.
Shortening $\g$ to $r>r_0$, we get $\int_\g e^{-B}\le \int_{r_0}^\infty r^{-\a_0}<\infty$
proving incompleteness of $g$.

Suppose $\a(u)\in (0,1)$. Fix $\a_1\in (\a(u),1)$ and any smooth path $\s$ going to infinity.
Find $r_1$ with $M(r)<\a_1\log r$ for all $r>r_1$. Now
$u(re^{i\theta})\le M(r)$ implies
\[
\int_\s e^{-u}ds\ge\int_{r_1}^\infty r^{-\a_1} dr=\infty
\]  
so $g$ is complete.
\end{proof}

\begin{prop}
\label{prop: harmonic}
If $u$ is harmonic and nonconstant, then $e^{-2u}g_0$ is not complete and $\a(u)=\infty$.
\end{prop}
\begin{proof}
If $\a(u)$ is finite, then by rescaling we may assume that $\a(u)<1$
so that $e^{-2u}g_0$ is complete by Lemma~\ref{lem: complete > and <}.
Since $u$ is harmonic, $e^{-2u}g_0$ has zero curvature. Any two zero curvature metrics
on $\mathbb R^2$ are isometric, so
$e^{-2u}g_0=\psi^*g_0$ for some $\psi\in\mathrm{Diff}(\mathbb R^2)$.
It follows that $\psi$ is conformal, and hence $\psi$ or its composition
with the complex conjugation is affine.
Therefore $\psi^*g_0$ is a constant multiple of $g_0$, hence $u$
is constant.
\end{proof}

\begin{rmk}
There is a purely analytic proof of incompleteness:
since $u$ is harmonic it is the real part of an entire function $f$.
Thus $e^u$ is $|e^f|$ where $e^f$ is an entire function with no zeros,
and hence so is $e^{-f}$. Huber~\cite[Theorem 7]{Hub-1957} proves 
that there is a path going to infinity such that the integral of
$|e^{-f}|^{-1}=e^u$ over the path is finite. (Actually, Huber's result
applies to any non-polynomial entire function in place of $e^{-f}$,
and his proof is much simplified when the entire function has
finitely many zeros, as happens for $e^{-f}$; the case of 
finitely many zeros is explained on~\cite[page 71]{Kap}). 
\end{rmk}

\begin{lem}
\label{lem: S_alpha closed}
$S_\a$ is a closed convex subset in the Fr{\'e}chet space $C^\infty(\mathbb R^2)$.
\end{lem}
\begin{proof} 
Convexity is immediate: If $u=s u_1+(1-s)u_0$ with $s\in [0,1]$
and $u_j\in S_\a$, then $u$ is subharmonic, and
$M(r, u)\le sM(r, u_1)+(1-s)M(r, u_0)$, so
dividing by $\log r$ and taking $r\to\infty$
yields $\a(u)\le s \a(u_1)+(1-s) \a(u_0)\le \a$. 

Fix $u_j\in S_\a$ and a smooth function $u$
such that $u_j\to u$ in $C^\infty(\mathbb R^2)$.
Clearly $u$ is subharmonic, and also
$M(r, u_j)\to M(r, u)$ for each $r$, so 
$\mu(\cdot , u_j)\to \mu(\cdot, u)$ pointwise. 
Set $\mu_j:=\mu(\cdot, u_j)$, $\mu:=\mu(\cdot, u)$. 
Since $\mu_j$, $\mu$ are convex, $\mu_j^\prime$, $\mu^\prime$ 
exist outside a countable subset $\Sigma$ \cite[page 7, Theorem C]{RobVar}, and
the convergence $\mu_j\to \mu$
is uniform on compact sets \cite[page 17, Theorem E]{RobVar}, 
which easily implies that
$\mu_j^\prime\to \mu^\prime$ outside $\Sigma$~\cite[Exercise C(9), page 20]{RobVar}.  

By convexity $\mu_j^\prime$, $\mu^\prime$
are non-decreasing outside $\Sigma$~\cite[page 5, Theorem B]{RobVar}.
It follows that $\mu_j^\prime\le\a$ outside $\Sigma$ for if 
$\mu_j^\prime(t_1)=\a_1>\a$, then $\mu_j^\prime\ge\a_1$ 
for $t\ge t_1$, so integrating we get $\mu_j(t)-\mu_j(t_1)\ge \a_1(t-t_1)$
which contradicts 
\[
\lim_{t\to \infty}\frac{\mu_j(t)}{t}\le \a.
\] 
Since $\mu_j^\prime\to\mu^\prime$ we get $\mu^\prime\le\a$ outside $\Sigma$. 
Integrating gives $\mu(t)\le\a (t-t_0)+\mu(t_0)$ everywhere,
so $u\in S_\a$.
\end{proof}

\begin{lem} 
\label{lem: S_alpha}
If $\a>0$, then $S_\a$ is not locally compact.
\end{lem} 
\begin{proof} 
Let us fix any $u\in S_\a$ with $0<\a(u)< \infty$, and suppose
arguing by contradiction that $u$ has a
compact neighborhood $V$ in $S_\a$.
Since $u$ is not harmonic, there is a closed disk 
$D\subset \mathbb R^2$ such that $\triangle u>0$ on $D$. 
The Fr\'echet space $C^\infty(D)$ is not locally compact
because it is an infinite dimensional topological vector space~\cite[Theorem 9.2]{Tre-TVS-book}.
Homogeneity of $C^\infty(D)$ implies that 
it contains no compact neighborhood.
To get a contradiction we show that 
the restriction map $\delta\co C^\infty(\mathbb R^2)\to C^\infty(D)$
takes $V$ to a compact neighborhood of $u\vert_D$. 
Compactness of $\delta(V)$ follows from the continuity of $\delta$. 
If  $\delta(V)$ were not a neighborhood of $u\vert_D$, there 
would exist a sequence $u_i\to u$ in $C^\infty(\mathbb R^2)$
with $u_i\vert_D\in C^\infty(D)\setminus\delta(V)$.
Let $\phi$ be a bump function
with $\phi\vert_D\equiv 1$ and such that the support of $\phi$ lies in a compact
neighborhood of $D$ on which $\triangle u>0$.
Then $u+\phi(u_i-u)$ converges to $u$ and only differs from $u$
on a compact set where $\triangle u>0$. Thus $u+\phi(u_i-u)\in V$
for large $i$, and $u+\phi(u_i-u)\vert_D=u_i\vert_D$,
which contradicts the assumption $u_i\vert_D\not\in\delta(V)$.
\end{proof}

\begin{thm}
$S_1$ equals the set of smooth subharmonic functions $u$
such that the metric $e^{-2u}g_0$ is complete.
\end{thm}
\begin{proof}
If $u\notin S_1$, then $e^{-2u}g_0$ is incomplete by 
Lemma~\ref{lem: complete > and <}. Suppose $u\in S_1$
while $e^{-2u}g_0$ is incomplete, and aim for a contradiction. 
By incompleteness of $e^{-2u}g_0$ there is a smooth path $\g$ in $\mathbb R^2$ 
going to infinity such that $\int_\g e^{-u}ds<\infty$. Now $u\le M$ implies
$\int_\g e^{-M}ds<\infty$. It is convenient to replace $M$, $\mu$ with
nearby smooth functions with similar properties which is possible by
a result of Azagra~\cite{Aza} that there is a smooth convex function 
$\nu$ defined on $\mathbb R$ such that $\mu-1\le \nu\le\mu$. 
By Lemma~\ref{lem: complete > and <} we have $\a(u)=1$, so 
$\frac{\nu(t)}{t}\to 1$ as $t\to\infty$.

For $r>0$ set $N(r):=\nu(\log r)$; the function $(x,y)\to N(r)$ is
subharmonic: 
\[
\triangle N=\nu^{\prime\prime}(t_x^2+t_y^2)+\nu^\prime\triangle t=\nu^{\prime\prime}r^{-2}.
\]
Here $t_x$, $t_y$ are partial derivatives of $t=\log r$; note that
$\triangle t=0$ while $t_x=\frac{x}{r^2}$ and $t_y=\frac{y}{r^2}$.

Set $d(r):= \frac{N(r)}{\log r}-1$ so that $e^{-N(r)}=r^{-1-d(r)}$. 
Since $u\in S_1$ and  $e^{-2u}g_0$ is incomplete, 
Lemma~\ref{lem: complete > and <} implies that
$u\notin S_\a$ for $\a<1$ so that
$d(r)\to 0$ as $r\to\infty$.
Also $M-1\le N\le M$ so that $\int_\g r^{-1-d}ds=\int_\g e^{-N}ds<\infty$.

In deriving a contradiction it helps consider the following cases.

If $d^{\,\prime}\ge 0$ for all large $r$, then since $d$ converges to zero as $r\to \infty$,
we must have $d\le 0$ for large $r$, so after shortening $\g$ we get
\[
\int_\g r^{-1-d}ds\ge \int_\g r^{-1}ds
=\infty\]
which is a contradiction.

If $d^\prime$ is neither nonnegative nor nonpositive as $r\to\infty$, then there is a point where 
$d^\prime$ and $d^{\prime\prime}$ are both negative, which contradicts subharmonicity of $N$ for $r>1$
as \begin{equation}
\label{form: N subharmonic via d}
0\le \triangle N=N^{\prime\prime}+\frac{N^\prime}{r}=
d^{\prime\prime}\log r+d^\prime\left(\frac{\log r}{r}+\frac{2}{r}\right).\end{equation}

It remains to deal with the case when $d^\prime \le 0$ for large $r$. Multiply
(\ref{form: N subharmonic via d}) by $r\log r$ to get
\[
0\le r\log^2 r\left(d^{\prime\prime}+d^\prime\left(\frac{1}{r}+\frac{2}{r\log r}\right)\right)=
\left( r(\log r)^2 d^\prime \right)^\prime
\]
which integrates over $[\rho, r]$ to $d^\prime (\rho)\, \rho\log^2\! \rho\le d^\prime (r)\, r\log^2\! r$.
Since $d^\prime \le 0$ for all large $r$, we conclude that
$c:=d^\prime (\rho ) \rho\log^2\! \rho$ is a nonpositive constant.
Set $f(r):=-\frac{c}{\log r}$ so that $f^\prime=\frac{c}{r\log^2\! r}\le d^\prime$. 
Integrating $d^\prime-f^\prime\ge 0$ over $[r, R]$ gives $d(R)-f(R)\ge d(r)-f(r)$
and since $d(R)$, $f(R)$ tend to zero as $R\to \infty$, 
we get $d(r)\le f(r)$ for all large $r$. 
Hence $\int_\g r^{-1-d}ds\ge \int_\g r^{-1-f}ds=e^c\int_\g r^{-1}ds=\infty$
which again is a contradiction.
\end{proof}

\section{Loose ends}
\label{sec: misc}

In this section we justify some straightforward claims made in the introduction.

\begin{lem}
\label{lem: conformal unique}
If $g$ is conformal to $g_0$, then there are unique 
$\varphi\in\mathrm{Diff}^+_{0,1}(\mathbb R^2)$ and $v\in C^\infty(\mathbb R^2)$
such that $g$ equals $\varphi^*e^{v} g_0$.
\end{lem}
\begin{proof}
Any metric $g$ conformal to $g_0$ can be written as $\psi^* e^{f}g_0=e^{f\circ\psi}\psi^*g_0$
where $\psi\in\mathrm{Diff}(\mathbb R^2)$ and $f\in C^\infty(\mathbb R^2)$.
To prove existence recall that any diffeomorphism of $\mathbb R^2$
is isotopic either to the identity or to the reflection $z\to\bar z$. 
The metric $\psi^*g_0$  
does not change when we postcompose $\psi$ with an isometry
of $g_0$, and postcomposing $\psi$ with an affine map $z\to az+b$, $a,b\in\mathbb C$ results in rescaling
which can be subsumed into $f\circ\psi$ changing it by an additive constant,
see the proof of Theorem~\ref{thm: mod disconnected}.
Thus composing with $z\to\bar z$ if needed, and with an affine map we can
arrange $\psi$ to lie in $\mathrm{Diff}^+_{0,1}(\mathbb R^2)$.

To see uniqueness rewrite $\varphi_1^*e^{v_1} g_0=\varphi_2^* e^{v_2} g_0$ as
\begin{equation}
\label{form: conformal aut}
(\varphi_1\circ\varphi_2^{-1})^* g_0=e^{v_2-v_1\circ \varphi_1\circ\varphi_2^{-1}} g_0
\end{equation}
where $\varphi_1, \varphi_2\in \mathrm{Diff}^+_{0,1}(\mathbb R^2)$. Hence 
$\varphi_1\circ\varphi_2^{-1}$ is a conformal automorphism of $g_0$
that preserves orientation and fixes $0$, $1$. It follows that $\varphi_1\circ\varphi_2^{-1}$
is the identity, see the proof of Theorem~\ref{thm: mod disconnected}.
Hence (\ref{form: conformal aut}) implies $v_1=v_2$.
\end{proof}

Let us think of $\mathcal R^{k}_{\ge 0}(\mathbb R^2)$ as a subset
of $C^\infty(\mathbb R^2,\mathrm{Bil}(\mathbb R^2))$, where 
$\mathrm{Bil}(\mathbb R^2)$ is the vector space of
real bilinear forms on $\mathbb R^2$. Since a subspace of a separable metrizable 
space is separable and metrizable, the following lemma implies that 
$\mathcal R^k_{\ge 0}(\mathbb R^2)$ is separable and metrizable, and 
$\mathcal R^\infty_{\ge 0}(\mathbb R^2)$ is completely metrizable.

\begin{lem}
\label{lem: separable} Let $N$, $M$ be smooth manifolds and $0\le k\le \infty$.
With the $C^k$ topology the space $C^\infty(M,N)$ is 
\newline\textup{(1)} separable
and metrizable, 
\newline\textup{(2)}
completely metrizable if $k=\infty$,
\newline\textup{(3)}
a Fr\'echet space if $N$ is a Euclidean space and $k=\infty$.
\end{lem}
\begin{proof}
The space $C^\infty(M,N)$ sits in $C^k(M,N)$ which embeds as a closed subset into 
$C^0(M, J^k(M,N))$
where $J^k(M,N)$ is the space of $k$-jets which is a $C^0$ 
manifold, see e.g.~\cite[Section 2.4]{Hir}. 
Separability is implied
by having a countable basis, and since the latter property is inherited by
subspaces it suffices to show that $C^0(M, J^k(M,\mathbb R))$
has a countable basis, but in general if the spaces  
$X$, $Y$ have a countable basis and and if $X$ is locally compact,
then $C^0(X,Y)$ with the compact-open topology has a countable 
basis~\cite[Theorem XII.5.2]{Dug}. Similarly, metrizability
is inherited by subspaces, and complete metrizability is inherited
by closed subspaces, while $C^0(X,Y)$ with the compact-open topology
is completely metrizable whenever $X$ is locally compact and $Y$ is 
completely metrizable~\cite[Theorem 2.4.1]{Hir}. 
A proof of (3) can be found in~\cite[Example 10.I]{Tre-TVS-book}.
\end{proof}

\begin{prop}
\label{prop: cone, not homeo}
There is an isometry class in $\mathcal M^{k}_{\ge 0}(\mathbb R^2)$ that
lies in every neighborhood of the isometry class of $g_0$.
\end{prop}
\begin{proof}
Let $f\co [0,\infty)\to [0,\infty)$ be a convex smooth 
function with $f(x)=2$ for $x\in [0,1]$ and $f(x)=x$ 
if $x\ge 3$.  
The surface of revolution in $\mathbb R^3$ obtained by rotating 
the curve $x\to (x, 0, f(x))$ about the {$z$-axis} defines a complete metric
$g$ on $\mathbb R^2$ of nonnegative curvature.
For each $r$ there is a metric $r$-ball in
$(\mathbb R^2, g)$ that is isometric to the Euclidean $r$-ball $B_r$
about the origin in $(\mathbb R^2, g_0)$. Extending the isometry
to a self-diffeomorphism $\varphi_r$ of $\mathbb R^2$ gives 
$\varphi_r^*g\vert_{B_r}=g_0\vert_{B_r}$
so  $\varphi_r^*g$ converge to $g_0$ uniformly on compact subsets 
as $r\to \infty$.
\end{proof}

\section{Beltrami equation: dependence of solutions on the dilatation}

\begin{thm} 
\label{thm: pi k homeo}
Let $k$ be a nonnegative integer or $k=\infty$ and $\g\in (0,1)$.
If $S_1$ is given the $C^{k+\g}$ topology and $\mathrm{Diff}^+_{0,1}(\mathbb C)$ 
is given the $C^{k+1+\g}$ topology, then the map $\Pi$ of \textup{(\ref{form: bijection})}
is a homeomorphism.
\end{thm}
\begin{proof}
By Theorem~\ref{thm: char of completeness} and Lemma~\ref{lem: conformal unique}
the map $\Pi(u,\varphi)=\varphi^*e^{-2u}g_0$ is a bijection. 
If $\varphi\in \mathrm{Diff}^+_{0,1}(\mathbb C)$ varies
in the $C^{k+1+\g}$ topology, the its differential varies in the $C^{k+\g}$ topology,
which implies continuity of $\Pi$. It remains to show that $\Pi^{-1}$
is continuous.

Write $g=\varphi^*e^{f}g_0=e^{f\circ \varphi}\varphi^*g_0$ with 
$\varphi$ is an orientation-preserving diffeomorphism of $\mathbb C$
that fixes $0$ and $1$.
The Jacobian of $\varphi$
equals $|\varphi_z|^2-|\varphi_{\bar z}|^2$, and since $\varphi$ is orientation-preserving,
we get 
$\frac{|\varphi_{\bar z}|}{|\varphi_z|}<1$.
Computing
\[
\frac{\varphi^*g_0}{|\varphi_z|^2}=
\frac{|d\varphi|^2}{|\varphi_z|^2}=
\left|dz+\frac{\varphi_{\bar z}}{\varphi_z} d\bar z\right|^2.
\] 
gives $g=e^{f\circ \varphi}|\varphi_z|^2 \left|dz+\frac{\varphi_{\bar z}}{\varphi_z} d\bar z\right|^2$.
Also we can write $g=Edx^2+2Fdxdy+Gdy^2$
as $\l |dz+\mu d\bar z|^2$, where \[
\l=\frac{1}{4}(E+G+2\sqrt{EG-F^2})\quad\text{and}\quad\mu=\frac{E-G+2iF}{4\l}.\]
Positive definiteness of $g$ easily implies $|\mu|<1$ and $\l>0$.
Note that $\mu$ and $\l$ depend smoothly on $g$.
 
Comparing the two descriptions of $g$ we see that $\varphi_{\bar z}=\mu \varphi_z$, that is,
$\varphi$ solves the Beltrami equation with dilatation $\mu$. Furthermore, $\l=e^{f\circ \varphi}|\varphi_z|^2$
so that $f=\log (\l |\varphi_z|^{-2})\circ \varphi^{-1}$.

Since $\mathcal R^{k+\g}_{\ge 0}(\mathbb R^2)$ is metrizable, the continuity of $\Pi^{-1}$
would follow once we show that for any sequence of metrics 
$g_l=\varphi_l^*e^{f_l} g_0=\l_l|dz+\mu_l d\bar z|^2$,
$\varphi_l\in\mathrm{Diff}^{+}_{0,1}(\mathbb C)$
that converges to $g$ uniformly on compact subsets in the $C^{k+\g}$ topology, the
maps $\varphi_l$, $f_l$ converge to $\varphi$, $f$ in the $C^{k+1+\g}$, $C^{k+\g}$
topology, respectively.
A key ingredient is the smooth dependence of $\varphi$ on $\mu$
established by Earle-Schatz in~\cite{EarSch}. 

To state their result let $U$, $U^\prime$ be domains
in $S^2$ whose boundaries are embedded circles, and let
$a_1$, $a_2$, $a_3$ and $a_1^\prime$, $a_2^\prime$, $a_3^\prime$ be two 
triples of distinct points on $\d U$ and $\d U^\prime$ respectively.
Recall that given a $C^\infty$ function $\nu\co U\to\mathbb C$
with $|\nu|\le k<1$ for some constant $k$, there is a 
unique homeomorphism $w^\nu\co \bar U\to\bar U^\prime$
that restricts to a diffeomorphism $U\to U^\prime$, maps each $a_k$ to $a_k^\prime$, 
and solves the Beltrami equation with dilatation $\nu$, 
see e.g.~\cite[p.183, 194]{LehVir} for existence and uniqueness, 
and~\cite[Theorem 2.2 in Section 4 of Chapter 2]{Vek-book} for regularity.  

The Continuity Theorem of Earle-Schatz~\cite[p. 181]{EarSch} 
states that varying $\nu$ in the $C^{k+\g}$ topology results in varying 
$w^\nu$ in $C^{k+1+\g}$ topology. Strictly speaking,
Earle-Schatz assume that $U$, $U^\prime$ equal the upper half plane, and
both triples of points equal $0$, $1$, $\infty$, but 
the conformal invariance of 
the Beltrami dilatation, see e.g.~\cite[formulas (7)-(8) on p.9]{Ahl-book},
together with the  Riemann mapping theorem  
give the same conclusion for any $U$, $U^\prime$ 
as above.

The Continuity Theorem does not immediately apply in
our setting, where $U=\mathbb C=U^\prime$ and $|\nu|$ is not
bounded way from $1$. Instead we use the
theorem locally, on an arbitrary disk $B_t=\{z\in\mathbb C\,:\, |z|<t\}$,
but then the difficulty is that the domain $\varphi(B_t)$ 
may change as the diffeomorphism $\varphi$ varies with $\mu$.
Below we resolve the issue by adjusting $\varphi(B_t)$ via an ambient diffeomorphism
that is the identity on a given compact set. Exhausting $\mathbb C$
by such compact sets yields the smooth dependence of $\varphi$ on $\mu$.

Let $K$ be a compact subset of $\mathbb C$. Let 
$\tilde g=\tilde\varphi^*e^{\tilde f}g_0=\tilde\l|dz+\tilde\mu d\bar z|^2$
be a metric that is $C^{k+\g}$ close
to $g=\varphi^*e^{f}g_0=\l|dz+\mu d\bar z|^2$ over $K$, where 
$\tilde\varphi, \varphi\in \mathrm{Diff}^{+}_{0,1}(\mathbb C)$. 
Choose $s$ with $\tilde\varphi(K)\subset B_{s}$.
The domains $\varphi(B_r)$, $r>0$ exhaust $\mathbb C$, and so do the domains
$\tilde\varphi(B_r)$, which allows us to find $r$
with $\bar B_{s}\subset \varphi(B_r)\cap \tilde\varphi(B_r)$.

It is easy to construct an orientation-preserving diffeomorphism $h$ of $\mathbb C$
that maps $\tilde\varphi(\bar B_r)$ onto $\varphi(\bar B_r)$, equals the identity
on $B_{s}$, and has the property that $h\circ\tilde\varphi$ and $\varphi$
agree at the points $-r$, $ir$, $r$ of $\d B_r$.
(Indeed, $\tilde\varphi(\d B_r)$, $\varphi(\d B_r)$ are homotopic smooth simple closed curves
in the open annulus $\mathbb C-\bar{B}_{s}$, and hence they can be moved to each 
other by a compactly supported ambient isotopy of the annulus, 
see e.g.~\cite[Propositions 1.10-1.11]{FarMar}. The isotopy starts at the identity and ends
at a diffeomorphism that has the desired property except it needs to be
adjusted at three points.
The identity component of $\mathrm{Diff}(S^1)$ acts transitively on the set of
triples of distinct points of $S^1$, e.g.
if $S^1$ is identified with the boundary 
of the upper half plane, then 
the map $x\to \frac{(x-a)(c-b)}{(x-b)(c-a)}$
takes $a$, $b$, $c$ to $0$, $\infty$, $1$, respectively, and preserves an orientation, and hence
is isotopic to the identity of $S^1$. 
So given two triples of points in $S^1\times \{0\}$
there is a compactly supported orientation-preserving diffeomorphism of $S^1\times \mathbb R$ that 
takes one triple to the other one. Here we identify 
$S^1\times \mathbb R$, $S^1\times \{0\}$ with $\mathbb C-\bar{B}_{s}$, 
$\varphi(\d B_r)$, respectively. 
Composing the two diffeomorphisms, and extending the result 
by the identity on $\bar B_s$ yields the desired $h$).

Since $g\vert_K$, $\tilde g\vert_K$ are $C^{k+\g}$ close, so are the dilatations
of $\varphi\vert_K$, $\tilde\varphi\vert_K=h\circ\tilde\varphi\vert_K$.
Thus $\varphi\vert_{B_r}$, $h\circ\tilde\varphi\vert_{B_r}$ are 
diffeomorphisms of $B_r$ onto $\varphi({B_r})$ whose dilatations
are $C^{k+\g}$ close on $K$. The absolute values of the dilatations
are less than $1$ (as the diffeomorphisms are orientation-preserving), and
hence are bounded away from $1$ by compactness of $\bar B_r$.
Now the Continuity Theorem implies that
$\varphi\vert_{B_r}$, $h\circ\tilde\varphi\vert_{B_r}$
are $C^{k+1+\g}$ close over $K$. 
It follows that $\l|\varphi_z|^{-2}$, $\tilde\l|\tilde\varphi_z|^{-2}$ are $C^{k+\g}$ close over $K$.

Thus if a sequence of metrics $g_l$ converges to $g$ uniformly 
on compact subsets in the $C^{k+\g}$ topology, then 
the corresponding diffeomorphisms $\varphi_l$ converge to $\varphi$ uniformly 
on compact subsets in the $C^{k+1+\g}$ topology. Since $\mathrm{Diff}^{+}_{0,1}(\mathbb C)$
with the $C^{k+1+\g}$ topology is a topological group we also have
the $C^{k+1+\g}$ convergence of $\varphi_l^{-1}$ to $\varphi^{-1}$. This implies that 
$\log (\l_l |(\varphi_l)_z|^{-2})\circ \varphi_l^{-1}$ converges to 
$\log (\l |\varphi_z|^{-2})\circ \varphi^{-1}$
in the $C^{k+\g}$ topology, which completes the proof.
\end{proof}

\section{Metric deformation with obstacles}
\label{sec: applications}

\begin{proof}[Proof of Theorem~\ref{thm: space of metrics countable}]
As was explained in the introduction,
$S_1\times\mathrm{Diff}^+_{0,1}(\mathbb R^2)$ is homeomorphic to $\ell^2$,
and since $\Pi^{-1}(K)$ is a countable union of compact sets,
its complement is homeomorphic to $\ell^2$, which contains the Hilbert cube
and hence an embedded copy of $X$. Any two pair of points of $\ell^2$ can
be moved to each other by an affine self-homeomorphism of $\ell^2$, so postcomposing 
with such a homeomorphism one can ensure that the embedding $X\to\ell^2$ maps $x_1, x_2$ to 
$\Pi^{-1}(g_1)$, $\Pi^{-1}(g_2)$, respectively. Since $X$ sits in
an embedded copy of the Hilbert cube (a compact set), and since
$\mathcal R^{k}_{\ge 0}(\mathbb R^2)$ is Hausdorff, the restriction
of $\Pi$ to the embedded copy of $X$ is a homeomorphism onto
its image, which has desired properties.
\end{proof}

\begin{proof}[Proof of Theorem~\ref{thm: space of metrics finite dim}]
Let $S$ be a finite dimensional subspace of $\mathcal R^{k}_{\ge 0}(\mathbb R^2)$.
Fix two points $g_1$, $g_2$ in the complement of $S$. 
Theorem~\ref{thm: space of metrics countable} implies 
$g_1$, $g_2$ lie in a subspace $X$ of
$\mathcal R^{k}_{\ge 0}(\mathbb R^2)$ that
is homeomorphic to $\mathbb R^n$ with $n\ge\dim(S)+2$.
Since $S\cap X$ has dimension $\le\dim(S)$~\cite[Theorem 1.1.2]{Eng-book}, 
its codimension in $X\cong\mathbb R^n$ is $\ge 2$, hence the points
$g_1$, $g_2$ lie is a continuum in $X$ that is disjoint from 
$S$~\cite[Theorem 1.8.19]{Eng-book}.

If $S$ is closed one can say more: Suppose that $X$ is homeomorphic to $S^n$
with $n\ge\dim(S)+2$, 
so $S\cap X$ is a closed subset of $X$ 
of codimension $\ge 2$. By the cohomological characterization of
dimension, see~\cite[p.95]{Eng-book}, the space $S\cap X$ has trivial {\v C}ech
cohomology in dimensions $>n-2$, hence 
by the Alexander duality $X\smallsetminus S$
is path-connected, giving a path in $X\smallsetminus S$ joining
$g_1, g_2$. 
\end{proof}

\begin{proof}[Proof of Theorem~\ref{thm: mod disconnected}]
Let $q\co S_1\to \mathcal M^{k}_{\ge 0}(\mathbb R^2)$ denote the continuous
surjection sending $u$ to the isometry class of $e^{-2u}g_0$.

If $S$ is countable, it suffices to show that every fiber of $q$
is the union of countably many compact sets because then
the complement of a countable subset
in $\mathcal M^{k}_{\ge 0}(\mathbb R^2)$ is the image of $S_1$ with a countable collection
of compact subsets removed, which is homeomorphic to $\ell^2$~\cite[Theorem V.6.4]{BP-book},
and of course the continuous image of $\ell^2$ is path-connected.

A function $v\in S_1$ lies in the fiber over the isometry class of $e^{-2u}g_0$ 
if and only if $e^{-2v}g_0=\psi^*e^{-2u}g_0=e^{-2u\circ \psi}\psi^* g_0$
for some $\psi\in\mathrm{Diff}(\mathbb R^2)$. Note that
$\psi$ necessarily lies in $\mathrm{Conf}(g_0)$,
the group of conformal automorphisms of $g_0$, i.e.
either $\psi$ or $r\psi$ equals $z\to az+b$ for some $a,b\in \mathbb C$, 
where $a\neq 0$ and $r(z)=\bar z$.
Since $\psi^*g_0=|a|^2g_0$, we conclude that $v=u\circ\psi-\log|a|$. 
In summary, $v, u\in S_1$ satisfy $q(u)=q(v)$ if and only if
$v=u\circ\psi-\log|a|$ where either $\psi$ or $r\psi$ equals $z\to az+b$
with $a\neq 0$. 
Setting $\psi^{-1}[u]:=u\circ\psi-\log|a|$ 
defines a $\mathrm{Conf}(g_0)$-action on $C^{\infty}(\mathbb R^2)$
whose orbit through $u\in S_1$ contain the $q$-preimage of $q(u)$.

Now suppose that $S$ is closed, metrizable and finite dimensional.
Let $\hat S$ be the $q$-preimage of $S$. 
Fix two points $g_1$, $g_2$ 
in the complement of $S$, which are $q$-images of $u_1, u_2\in S_1\setminus \hat S$, respectively.
Since $S_1$ is homeomorphic to $\ell^2$, 
the proof of Theorem~\ref{thm: space of metrics countable} 
shows that $u_1, u_2$ lie in an embedded copy
$\hat Q$ of the Hilbert cube.
It is enough to show that $\hat Q\cap \hat S$ is finite dimensional
in which case $\hat Q\smallsetminus \hat S$ is path-connected because
the complement to 
any finite dimensional closed subset of the Hilbert cube is acyclic~\cite[Lemma 2.1]{Kro}.
Since $S$ is closed, $\hat Q\cap\hat S$ is compact, so
the restriction of $q$ to $\hat Q\cap\hat S$
is a continuous surjection $\hat q\co \hat Q\cap\hat S\to q(\hat Q)\cap S$
of compact separable metrizable spaces,
and in particular a closed map, which is essential for what follows. 

Since the target of $\hat q$ lies in $S$, it is finite dimensional.
The map $\hat q$ is closed, so finite-dimensionality of the domain of $\hat q$ would follow from
a uniform upper bound on the dimension of the fibers of $\hat q$~\cite[Theorem 1.12.4]{Eng-book}.
Each fiber lies in the image of the orbit 
map $\mathrm{Conf}(g_0)\to C^\infty(\mathbb R^2)$ described above.
If $G:=\mathrm{Conf}(g_0)$ and $G_u$ is the isotropy subgroup of $u$ in $G$,
then the $G$-orbit of $u$ is the image of a one-to-one continuous map $o\co G/G_u\to C^\infty(\mathbb R^2)$.
Since $G/G_u$ is a manifold of dimension $l\le \dim G$, 
it is the union of countably many $l$-dimensional
compact domains. Restricting $o$ to each such domain
is a continuous one-to-one map from a compact space to a Hausdorff space,
which is a homeomorphism, and hence
the dimension of its image is $l$.
By the sum theorem for the dimension of a countable union of
closed subsets~\cite[Theorem 1.5.4]{Eng-book} the image of $o$
has dimension $l$. 
\end{proof}

\begin{rmk}
\label{rmk: flat metrics}
Theorems~\ref{thm: space of metrics countable}--\ref{thm: mod disconnected}
yield a deformation between any two given metrics $g_1, g_2$ that runs in 
a separable metrizable space, a continuum, or a path, and we now 
show that this deformation can be arranged to bypass any given set of complete
flat metrics $\mathcal F$. We do so in the setting of Theorem~\ref{thm: space of metrics countable}; 
the other two proofs are similar. Set 
$P_\a:=S_\a\times \mathrm{Diff}^+_{0,1}(\mathbb R^2)$, which we
identify with $\ell^2$. 
Flat metrics are parametrized by $P_0$ which is a closed linear
subset of infinite codimension in $P_1$.
If $P_\a^\prime:=P_\a\smallsetminus\{\Pi^{-1}(g_1), \Pi^{-1}(g_2)\}$,
then $P_0^\prime$ has property $Z$ in $P_1^\prime$, 
see~\cite[Lemma 1]{AHW}, so that~\cite[Theorem 3]{AHW} implies that
$P_1\smallsetminus P_0^\prime$ is homeomorphic to $P_1$, after which the proof is finished
as in Theorem~\ref{thm: space of metrics countable}.
\end{rmk}

\small
\bibliographystyle{plain}
\bibliography{modr2}

\end{document}